\documentclass[11pt]{article}

\usepackage[letterpaper,margin=1.25in]{geometry}

\usepackage{amsmath,amssymb,amsfonts,amsthm}
\usepackage{mathtools}
\usepackage{mathrsfs}
\usepackage{appendix}

\usepackage{enumitem}
\usepackage{xcolor}
\usepackage[
    colorlinks=true,
    linkcolor=blue,
    citecolor=blue,
    urlcolor=blue
]{hyperref}

\numberwithin{equation}{section}
\allowdisplaybreaks

\theoremstyle{plain}
\newtheorem{theorem}{Theorem}[section]
\newtheorem{proposition}[theorem]{Proposition}
\newtheorem{lemma}[theorem]{Lemma}
\newtheorem{corollary}[theorem]{Corollary}

\theoremstyle{definition}

\theoremstyle{remark}
\newtheorem{remark}{Remark}

\title{Liouville-type theorems for the stationary fractional Navier--Stokes equations in $\mathbb{R}^n$}

\author{Jihoon Lee and Juhyeong Lee}

\date{}

\begin{document}

\maketitle

\begin{abstract}
We establish Liouville-type theorems for the stationary fractional Navier--Stokes equations in $\mathbb{R}^n$ under suitable integrability conditions on the velocity field $u$ and a large-scale Morrey-type bound on the fractional energy. As a corollary, these assumptions are automatically satisfied if $u \in \dot{H}^{\frac{\alpha}{2}}(\mathbb{R}^n)$, yielding Liouville-type results under the finite fractional energy condition for $\frac{n}{3} \le \alpha < \frac{n+2}{3}$, where $\alpha$ denotes the order of the fractional Laplacian $(-\Delta)^{\frac{\alpha}{2}}$. This range reflects a scaling-critical correspondence between Liouville-type theorems in the finite-energy setting and the threshold arising in partial regularity theory. The proof relies on direct kernel estimates for the commutator of the fractional Laplacian, based on a dyadic decomposition of the tail term, which remain valid in the hyper-dissipative case. The argument also uses a bootstrap argument that propagates integrability from near the scaling-invariant exponent down to lower exponents, including the Sobolev embedding exponent.

\medskip \noindent\textbf{Keywords:} Liouville-type theorems, fractional Laplacian, Navier--Stokes equations

\medskip \noindent\textbf{2020 Mathematics Subject Classifications:} 35B53, 35R11, 76D05, 76W05
\end{abstract}

\section{Introduction}

In this paper, we study the stationary fractional Navier--Stokes equations in $\mathbb{R}^n$: 
\begin{equation} \label{eq:1.1} \tag{1.1}
\begin{cases}
(-\Delta)^{\frac{\alpha}{2}}u + (u \cdot \nabla) u +\nabla p  = 0, \\
\nabla\cdot u = 0,
\end{cases}
\end{equation}
with the uniform decay condition
\[
u(x)\to 0\quad \text{as} \quad |x|\to+\infty\ .
\]
Here, $u = (u_1(x),..., u_n(x))$ denotes the velocity field of the fluid and $p = p(x)$ denotes the pressure of the fluid. The fractional Laplacian $(-\Delta)^{\frac{\alpha}{2}}$ is defined as the Fourier multiplier with symbol $|\xi|^{\alpha}$, where $\alpha \in (0,4)$.

Before turning to our model, we recall the classical Liouville problem for the stationary Navier--Stokes equations in $\mathbb{R}^n$, corresponding to the case with $\alpha = 2$ in \eqref{eq:1.1}. It is unknown whether any smooth solution $u$ to 
\begin{equation} \label{eq:1.2} \tag{1.2}
-\Delta u + (u \cdot \nabla) u + \nabla p = 0, 
\quad \nabla \cdot u = 0, \quad
\lim_{|x|\to+\infty} u(x) = 0, \quad
 \int_{\mathbb{R}^n} |\nabla u|^2 \, dx < +\infty    
\end{equation}
must be trivial. This problem was introduced in \cite{21}, where Galdi established Liouville-type theorems for \eqref{eq:1.2} under the integrability conditions \(u \in L^{\frac{9}{2}}(\mathbb{R}^3)\) for $n=3$ and \(u \in L^{\frac{2n}{n-2}}(\mathbb{R}^n)\) for $n \ge 4$. Since $\dot H^1(\mathbb{R}^n)\hookrightarrow L^{\frac{2n}{n-2}}(\mathbb{R}^n)$, the condition for $n \ge 4$ corresponds to the finite Dirichlet energy, whereas this is no longer the case in dimension three.

Although the problem remains open in $\mathbb{R}^3$, several Liouville-type results have been obtained under various assumptions. Using the sign structure of the head pressure $Q:=p+\frac{1}{2}|u|^2$, Chae \cite{3} proved triviality under the condition $\Delta u \in L^{\frac{6}{5}}(\mathbb{R}^3)$, which has the same scaling as $\nabla u \in L^2(\mathbb{R}^3)$. Chae and Wolf \cite{4} obtained a logarithmic improvement of $L^{\frac{9}{2}}(\mathbb{R}^3)$ by replacing it with $\int_{\mathbb{R}^3} |u|^{\frac{9}{2}} \left\{\log\!\left(2+\frac{1}{|u|}\right)\right\}^{-1} dx < +\infty$. They also established Liouville-type results under suitable integrability conditions on $Q$ by exploiting the head pressure equation. Further Liouville-type results based on the head pressure were obtained in \cite{6,12,13}. On the other hand, Seregin \cite{28} established the result under $u \in L^6(\mathbb{R}^3)$ together with $u \in \mathrm{BMO}^{-1}(\mathbb{R}^3)$, which has the same scaling property as $L^3(\mathbb{R}^3)$. Motivated by this scaling structure, Seregin \cite{29} extended these criteria to homogeneous Morrey spaces $\dot{M}^{p,q}(\mathbb{R}^3)$, proving triviality under $u \in \dot{M}^{2,6}(\mathbb{R}^3) \cap \dot{M}^{\frac{3}{2},3}(\mathbb{R}^3)$. This was further extended by Chamorro, Jarrín, and Lemarié-Rieusset \cite{14}, who established triviality under the condition $u \in \dot{M}^{2,3}(\mathbb{R}^3) \cap \dot{M}^{2,q}(\mathbb{R}^3)$ with $3 < q < +\infty$, replacing $\dot{M}^{2,6}(\mathbb{R}^3)$ with the more flexible space $\dot{M}^{2,q}(\mathbb{R}^3)$. We refer to \cite{5,10,24,27,30,31,36} for several related results. 

The standard proof relies on a Caccioppoli-type inequality. For $R>1$, let $B_R$ denote the ball of radius $R$ centered at the origin and the annulus $A(R):=B_{2R}\setminus B_{R}$. The following estimate holds for any smooth solution $u$ of \eqref{eq:1.2}:
\begin{equation} \label{eq:1.3} \tag{1.3}
\int_{B_R} |\nabla u|^2\,dx 
\le CR^{n-2-\frac{2n}{p}}\|u\|_{L^p(A(R))}^2
+ CR^{n-1-\frac{3n}{q}}\|u\|_{L^q(A(R))}\,\|u\|_{L^{q}(\mathbb{R}^n)}^{2},
\end{equation}
where $2 \le p \le \tfrac{2n}{n-2}$ and $3 \le q \le \tfrac{3n}{n-1}$. In dimension $n=3$, the Sobolev embedding exponent $\frac{2n}{n-2}$ is strictly larger than the exponent $\tfrac{3n}{n-1}$ arising from the nonlinear term, which restricts the argument. For the fractional case \eqref{eq:1.1}, the Sobolev embedding exponent becomes $\tfrac{6}{3-\alpha}$, which does not exceed $\tfrac{9}{2}$ for $0<\alpha\le \tfrac{5}{3}$. This raises the question of whether the argument still works under this range of $\alpha$, and in particular, whether the exponent $\tfrac{3n}{n-1}$ still restricts the range of admissible $\alpha$ in higher dimensions.

Motivated by these questions, we consider the following scaling-critical Liouville-type problem for \eqref{eq:1.1}: Any smooth solution $u \in \dot H^{\frac{\alpha}{2}}(\mathbb{R}^n)$ of \eqref{eq:1.1} satisfying the uniform decay condition for $\frac{n}{3} \le \alpha \le \frac{n+2}{3}$ must be trivial.

Here, scaling-critical refers to the range of $\alpha$ arising in partial regularity theory proved in \cite{16,23,33,35}. Indeed, for $n=3,4,5$ and $\frac{n}{3}\le \alpha<\frac{n+2}{3}$ with $\alpha\neq 1$, any suitable weak solution $u\in \dot H^{\frac{\alpha}{2}}(\mathbb{R}^n)$ of \eqref{eq:1.1} is regular away from a relatively closed singular set $\mathrm{Sing}(u)$ satisfying
\[
\mathcal{H}^{\,n+2-3\alpha}(\mathrm{Sing}(u))=0.
\]
The lower bound $\alpha=\frac{n}{3}$ ensures the Sobolev embedding into $L^3$, which is necessary to control the nonlinear term. The upper bound $\alpha=\frac{n+2}{3}$ is the value at which the Sobolev embedding exponent coincides with the scaling-invariant exponent, corresponding to the scaling-invariant space $L^{\frac{n}{\alpha-1}}(\mathbb{R}^n)$ of \eqref{eq:1.1}. At this value of $\alpha$, one has
\[
\frac{2n}{n-\alpha} = \frac{n}{\alpha-1} = \frac{3n}{n-1}.
\]

This coincidence suggests that Liouville-type problem may naturally arise with the threshold $\alpha=\frac{n+2}{3}$. In partial regularity theory, one aims to obtain smallness of scaling-invariant quantities on sufficiently small balls under the natural energy assumption. In contrast, in Liouville-type problems, one exploits the scaling structure of the equation to control the large-scale asymptotic behavior of solutions under the natural energy assumption. Thus the same threshold appears, although the direction of the argument is reversed.

In the three-dimensional case, such scaling-critical Liouville-type results have been established in recent works. Yang \cite{40} proved the Liouville-type result under the assumption $u\in L^{\frac{9}{2}}(\mathbb{R}^3)$ for $\frac{5}{3}\le \alpha <2$, which covers the endpoint $\alpha=\frac{5}{3}$. Later, Zeng \cite{42} proved that any smooth solution $u\in \dot H^{\frac{\alpha}{2}}(\mathbb{R}^3)$ is trivial for $1\le \alpha \le \frac{5}{3}$ as a corollary of Liouville-type results for the fractional MHD system under anisotropic Lebesgue spaces. See also \cite{15,25,34,38,39} for related results in $\mathbb{R}^3$.

Building on these results, we establish scaling-critical Liouville-type results in dimensions $3\le n \le 6$ for $\frac{n}{3} \le \alpha < \frac{n+2}{3}$. To handle the nonlocal operator $(-\Delta)^{\frac{\alpha}{2}}$ in a localized setting, most existing works are based on the Caffarelli–Silvestre extension \cite{2}. In contrast, our approach relies on direct kernel estimates, which allows us to replace the Sobolev space assumption to a large-scale Morrey-type condition. In addition, we exploit the structure of the equation to establish a bootstrap argument that propagates integrability from near the scaling-invariant exponent down to lower exponents, including the Sobolev embedding exponent. Since the bootstrap method is derived from the structure of the equation and is carried out by using the Hardy--Littlewood--Sobolev inequality \cite{32}, it naturally yields scaling-critical results, but does not cover the endpoint case $\alpha=\frac{n+2}{3}$. Our method also extends to the hyper-dissipative regime $\alpha>2$, thereby covering the dimensions $n=5$ and $n=6$, for which the scaling-critical threshold $\frac{n+2}{3}$ exceeds 2.

Using the same notation $A(R):=B_{2R}\setminus B_R$ for $R>1$ from \eqref{eq:1.3}, we first state our main result for $1\le \alpha < 2$ in dimensions $n=3,4,5$.

\begin{theorem} \label{thm:1}
Let $n=3,4,5$, and let $u$ be a smooth solution of \eqref{eq:1.1}. Assume that for some $0 \le \lambda < \frac{\alpha}{2}$, 
\[
\limsup_{r\to\infty} r^{-\lambda}
\left(\int_{A(r)} \left|(-\Delta)^{\frac{\alpha}{4}}u\right|^2\,dx \right)^{\frac{1}{2}}<+\infty.
\]
Suppose further that one of the following conditions holds:
\[
\begin{aligned}
\text{(i)} \quad & \alpha = 1, 
&& u \in L^q(\mathbb{R}^n), 
&& 3 \le q < +\infty, \\
\text{(ii)} \quad & 1 < \alpha < \min\!\left\{\frac{n+2}{3},\,2\right\}, 
&& u \in L^q(\mathbb{R}^n), 
&& 3 \le q < \frac{n}{\alpha-1}.
\end{aligned}
\]
Then $u \equiv 0$.
\end{theorem}

We next state our result for the hyper-dissipative case in dimensions $n=5,6$.

\begin{theorem} \label{thm:2}
Let $n=5,6$, and let $u$ be a smooth solution of \eqref{eq:1.1}. Assume that for some $0 \le \lambda < \frac{\alpha}{2}$,
\[
\limsup_{r\to\infty} r^{-\lambda}
\left(\int_{A(r)} \left|(-\Delta)^{\frac{\alpha}{4}} u\right|^2 \, dx \right)^\frac{1}{2}
<+\infty.
\]
Suppose further that for $2<\alpha<\frac{n+2}{3}$,
\[
\nabla u \in L^{s}(\mathbb{R}^n), 
\quad 
2\le s < \frac{n}{\alpha}.
\]
Then $u\equiv0$.
\end{theorem}

\begin{remark}
In Propositions~\ref{prop:3} and~\ref{prop:4}, the large-scale Morrey-type condition yields the decay factor
$R^{-\frac{\alpha}{2}+\lambda}\|u\|_{L^2(B_{2R})}$ in the far-field commutator estimate. Combined with the bootstrap $L^2$-control of $u$, the tail contribution vanishes as $R\to\infty$ whenever $0\le\lambda<\frac{\alpha}{2}$. Moreover, in part of the admissible range of $\alpha$, $\lambda$ can be chosen at the scaling-critical value. Under the scaling $u_\ell(x)=\ell^{\alpha-1}u(\ell x)$, one has
\[
\limsup_{r\to\infty} r^{-\lambda} 
\left\|(-\Delta)^{\frac{\alpha}{4}}u_\ell\right\|_{L^2(A(r))} 
=
\ell^{\lambda-\lambda_c} \limsup_{r\to\infty} (\ell  r)^{-\lambda}
\left\|(-\Delta)^{\frac{\alpha}{4}}u\right\|_{L^2(A(\ell r))},
\]
where $\lambda_c = \frac{n+2-3\alpha}{2}$. If $\frac{n+2}{4}<\alpha<\frac{n+2}{3}$, then $\lambda_c<\frac{\alpha}{2}$, so the scaling-critical choice
$\lambda=\lambda_c$ is admissible.
\end{remark}

We briefly recall that for $1 \le p \le q < +\infty$, the homogeneous Morrey space $\dot M^{p,q}(\mathbb{R}^n)$ consists of all functions $f \in L^p_{\mathrm{loc}}(\mathbb{R}^n)$ such that
\[
\|f\|_{\dot M^{p,q}(\mathbb{R}^n)}
:= \sup_{x_0 \in \mathbb{R}^n,\; r>0}
r^{\frac{n}{q}-\frac{n}{p}}
\left(\int_{B(x_0,r)} |f(x)|^p \, dx\right)^{\frac{1}{p}} < +\infty.
\]
In terms of homogeneous Morrey spaces, our Morrey-type assumption is satisfied when $(-\Delta)^{\frac{\alpha}{4}}u \in \dot M^{2,\frac{2n}{n-2\lambda}}(\mathbb{R}^n)$. In particular, for $\lambda=0$, this reduces to the finite fractional energy case $(-\Delta)^{\frac{\alpha}{4}}u\in L^2(\mathbb R^n)$, that is, $u\in \dot H^{\frac{\alpha}{2}}(\mathbb R^n)$.

In Theorem~\ref{thm:1}, by the Sobolev embedding $\dot H^{\frac{\alpha}{2}}(\mathbb{R}^n)\hookrightarrow L^{\frac{2n}{n-\alpha}}(\mathbb{R}^n)$, we observe that
\[
\frac{2n}{n-\alpha}
\in
\left[3,\frac{n}{\alpha-1}\right), \quad n=3,4,5, \quad \frac{n}{3} \le \alpha < \frac{n+2}{3}, 
\]
and hence the Sobolev embedding exponent belongs to the admissible range of $q$.

In Theorem~\ref{thm:2}, using the Sobolev embedding $\dot H^{\frac{\alpha}{2}}(\mathbb{R}^n)\hookrightarrow \dot W^{1,\frac{2n}{n+2-\alpha}}(\mathbb{R}^n)$, we similarly obtain
\[
\frac{2n}{n+2-\alpha} \in \left[2,\frac{n}{\alpha}\right), \quad n=5,6, \quad 2 < \alpha < \frac{n+2}{3}.
\]
Moreover, in the classical case $\alpha=2$, it is known that $u \in \dot H^{1}(\mathbb{R}^n)$ implies $u\equiv0$ for $n\ge4$. Combining these results, we obtain the following corollary.

\begin{corollary} \label{cor:3}
Let $3\le n \le 6$. Any smooth solution $u$ of \eqref{eq:1.1} satisfying $u \in \dot H^{\frac{\alpha}{2}}(\mathbb{R}^n)$ must be trivial for $\frac{n}{3}\le\alpha<\frac{n+2}{3}$.    
\end{corollary}

\begin{remark}
Direct kernel estimates allow us to avoid extension formulations that depend on the range of $\alpha$. In particular, for $\alpha>2$, the standard Caffarelli--Silvestre extension is no longer directly applicable, and a higher-order extension is required; see \cite{41}. The direct approach therefore allows the fractional and hyper-dissipative cases to be treated in a similar manner. The range $0<\alpha<\frac {n}{3}$ is simpler from the viewpoint of integrability, where the Sobolev embedding exponent satisfies $\frac{2n}{n-\alpha}<3$. Since $u$ is smooth and satisfies the uniform decay condition, we have $u\in L^\infty(\mathbb R^n)$. Hence, by interpolation, we obtain the integrability needed to control the nonlinear term. At the endpoint, however, the bootstrap method no longer provides the integrability gain needed to close the argument, and a different approach appears to be required.
\end{remark}

\begin{remark}
In two dimensions, if the vorticity satisfies the uniform decay condition, then Liouville-type theorems are expected to hold independently of $\alpha$, by the same maximum principle argument as in the classical case \cite{22}. In six dimensions, although the corresponding partial regularity results are known for $2 \le \alpha < \frac{7}{3}$ in \cite{16,20}, the range $2 \le \alpha < \frac{8}{3}$ remains consistent with the scaling of the equation and agrees with our results. For higher dimensions $n \ge 7$, since $\frac{n}{3} > 2$, the admissible range of $\alpha$ is shifted away from the classical case $\alpha=2$ and requires increasingly strong dissipation as $n$ grows. Although our argument may extend up to dimension $n=10$, where the upper endpoint is $\alpha<\frac{n+2}{3}=4$, we restrict our analysis to $3 \le n \le 6$ to remain in a regime closely related to the classical case.
\end{remark}

The paper is organized as follows. In Section~2, we introduce homogeneous Sobolev spaces, the Riesz potential, and auxiliary lemmas used in the proof. In Section~3, we prove Theorems~\ref{thm:1} and~\ref{thm:2}. The key step is to establish direct kernel estimates for the fractional commutator terms, which are obtained in Propositions~\ref{prop:1} and~\ref{prop:3}. Analogous estimates for the hyper-dissipative case are derived in Propositions~\ref{prop:2} and~\ref{prop:4} by a similar argument. Using these results, we derive a Caccioppoli-type inequality in each case. Together with the bootstrap argument based on Lemma~\ref{lem:4}, this yields the required integrability condition and completes the proof. In the appendix, we extend our approach to the coupled MHD system in $\mathbb{R}^3$ and present the corresponding results.

\section{Preliminaries}
Let $\mathcal{S}$ denote the Schwartz space and $\mathcal{S}'$ its dual space. For $1<p<+\infty$ and $s>0$, the homogeneous Sobolev space $\dot W^{s,p}(\mathbb{R}^n)$ is defined by
\[
\dot W^{s,p}(\mathbb{R}^n)
:= \left\{ f \in \mathcal{S}'(\mathbb{R}^n) \;:\;
(-\Delta)^{\frac{s}{2}} f \in L^p(\mathbb{R}^n) \right\},
\]
equipped with the semi-norm
\[
\|f\|_{\dot W^{s,p}(\mathbb{R}^n)}
:= \|(-\Delta)^{\frac{s}{2}} f\|_{L^p(\mathbb{R}^n)}.
\]
In particular, when $p=2$, we write
\[
\dot W^{s,2}(\mathbb{R}^n)=\dot H^s(\mathbb{R}^n).
\]

For $0<\alpha<n$, we also define the Riesz potential $I_\alpha$ by
\[
(I_\alpha f)(x)=C(n,-\alpha)\int_{\mathbb{R}^n}\frac{f(y)}{|x-y|^{n-\alpha}}\,dy,
\]
where $C(n,s)$ denotes the normalizing constant given by
\[
C(n,s)
:=
\frac{2^s\Gamma\!\left(\frac{n+s}{2}\right)}
{\pi^{\frac n2}\left|\Gamma\!\left(-\frac{s}{2}\right)\right|},
\quad -n<s<2,\quad s \neq 0.
\]
Here, the Riesz potential is identified with the inverse fractional Laplacian $(-\Delta)^{-\frac{\alpha}{2}}$.

We recall the following integral representation of the fractional Laplacian.

\begin{lemma} [\cite{19}, Lemma~3.2] \label{lem:1}
Let $\alpha\in(0,2)$ and let $(-\Delta)^{\frac{\alpha}{2}}$ be the fractional Laplacian operator. Then, for any $u\in\mathcal{S}$,
\begin{equation*}
\begin{aligned}
(-\Delta)^{\frac{\alpha}{2}} u(x)
&=C(n,\alpha)\mathrm{P.V.}\int_{\mathbb{R}^n}
\frac{u(x)-u(y)}{|x-y|^{n+\alpha}}\,dy \\
&=-\frac{1}{2}C(n,\alpha)\mathrm{P.V.}\int_{\mathbb{R}^n}
\frac{u(x+y)+u(x-y)-2u(x)}{|y|^{n+\alpha}}\,dy,
\quad \forall\, x\in\mathbb{R}^n.    
\end{aligned}
\end{equation*}
\end{lemma}

The following iteration lemma will be used to control the local fractional energy.

\begin{lemma} [\cite{11}, Lemma~2.1] \label{lem:2}
Let $f(t)$ be a nonnegative bounded function defined on $[r_0,r_1]$, $r_0\ge0$.
Suppose that for $r_0\le t<s\le r_1$, we have
\[
f(t)\le
\left[
\sum_{i=1}^m \bigl(A_i (s-t)^{-\alpha_i}+B_i\bigr)
\right]
+\varepsilon f(s),
\]
where $A_i$, $B_i$, $\alpha_i$, and $\varepsilon$ are nonnegative constants with
$0\le \varepsilon<1$.
Then for all $r_0\le \rho<R\le r_1$ we have
\[
f(\rho)\le
C\left[
\sum_{i=1}^m \bigl(A_i (R-\rho)^{-\alpha_i}+B_i\bigr)
\right],
\]
where $C$ is a constant depending only on $\alpha_i$ and $\varepsilon$.
\end{lemma}

The following lemma is the classical Hardy--Littlewood--Sobolev inequality.

\begin{lemma}[\cite{32}, Chapter~V, Theorem~1] \label{lem:3}
Let $0<\alpha<n$ and let $1<r<q<+\infty$ where \(\frac{1}{q}=\frac{1}{r}-\frac{\alpha}{n}.\)
Then, there exists a constant $C=C(q,r)>0$ such that
\[
\|I_\alpha f\|_{L^q(\mathbb{R}^n)}
\le C\,\|f\|_{L^r(\mathbb{R}^n)}.
\]
\end{lemma}

We derive the following $L^q$ estimate for smooth solutions of \eqref{eq:1.1}, which will be used in the bootstrap argument.

\begin{lemma} \label{lem:4}
Let $u$ be a smooth solution of \eqref{eq:1.1} satisfying $u \in L^r(\mathbb{R}^n)$, where $r$ is specified below. Then, for $3\le n \le6$, the following estimates hold:
\begin{equation*}
\begin{aligned}
\text{(i)} \ &\alpha=1:  
&& \|u \|_{L^q(\mathbb{R}^n)} \le C\|u\|_{L^{r}(\mathbb{R}^n)}^2, \quad r=2q,
\quad 1<q<+\infty, \\
\text{(ii)} \ & 1<\alpha<\frac{n+2}{3}: 
&& \|u\|_{L^q(\mathbb{R}^n)} \le C\|u\|_{L^{r}(\mathbb{R}^n)}^2,
\quad \frac{1}{q}=\frac{2}{r}-\frac{\alpha-1}{n}, \quad 2<r<2q<+\infty.
\end{aligned}
\end{equation*}
\end{lemma}

\begin{proof}[Proof of Lemma~\ref{lem:4}]
Let $u$ be a smooth solution of \eqref{eq:1.1}. Then, we may write
\[
u= -(-\Delta)^{-\frac{\alpha}{2}}
\,\mathbb P \nabla\cdot(u\otimes u).
\]
When $\alpha=1$, since $\nabla(-\Delta)^{-\frac{1}{2}}=\mathcal R$ where $\mathcal R$ denotes the Riesz transforms, the $L^q$-boundedness of the Leray projection and the Riesz transforms yields
\[
\|u\|_{L^q(\mathbb{R}^n)} 
\le C\|u\otimes u\|_{L^q(\mathbb{R}^n)} \le C\|u\|_{L^{2q}(\mathbb{R}^n)}^2, \quad 1<q<+\infty.
\]
For \(1<\alpha<\frac{n+2}{3}\), since $\nabla(-\Delta)^{-\frac{\alpha}{2}} = \mathcal R(-\Delta)^{-\frac{\alpha-1}{2}}$, the $L^q$-boundedness of the Leray projection and the Riesz transforms yields
\[
\|u\|_{L^q(\mathbb{R}^n)}
\le C \|(-\Delta)^{\frac{1-\alpha}{2}}(u\otimes u)\|_{L^q(\mathbb{R}^n)},
\quad 1<q<+\infty.
\]
Using Lemma~\ref{lem:3}, we get
\[
\|u\|_{L^q(\mathbb{R}^n)} 
\le C\|u\otimes u\|_{L^\frac{r}{2}(\mathbb{R}^n)} \le C\|u\|_{L^{r}(\mathbb{R}^n)}^2,
\]
where $\frac{1}{q}=\frac{2}{r}-\frac{\alpha-1}{n}$ with $2<r<\frac{2n}{\alpha-1}$.
\end{proof}

\section{Proof of the main results}
Let $\phi \in C_c^\infty(\mathbb R^n)$ be a radial non-increasing function such that
\[
0 \le \phi \le 1, \quad 
\phi(x)=1 \ \text{for } |x|<1, \quad 
\phi(x)=0 \ \text{for } |x|\ge 2.
\]
For $R>1$, define the standard cut-off function 
\[
\phi_R(x):=\phi\!\left(\frac{x}{R}\right).
\]
Then $0\le \phi_R\le 1$, $\phi_R(x)=1$ for $|x|<R$, and $\phi_R(x)=0$ for $|x|\ge 2R$. Moreover,
\[
\|\nabla^k \phi_R\|_{L^\infty(\mathbb R^n)} \le C R^{-k}, \quad k\in\mathbb N,
\]
and $\operatorname{supp}(\nabla \phi_R)\subset A(R)$.

Multiplying equation~\eqref{eq:1.1} by $\phi_R^2 u$ and integrating it over $\mathbb R^n$, we get
\begin{equation} \label{eq:3.1} \tag{3.1}
\begin{aligned}
\int_{\mathbb R^n} \phi_R^2|(-\Delta)^{\frac{\alpha}{4}}u|^2 &=
\frac{1}{2} \int_{\mathbb R^n} |u|^2u \cdot \nabla (\phi_R^2) \, dx 
- \int_{\mathbb R^n} pu \cdot \nabla (\phi_R^2) \, dx \\
&\quad -\int_{\mathbb R^n} (-\Delta)^{\frac{\alpha}{4}} u \cdot [(-\Delta)^{\frac{\alpha}{4}},\phi_R^2]u\, dx =: I_1+I_2+I_3,
\end{aligned}
\end{equation}
where 
\[
[(-\Delta)^{\frac{\alpha}{4}},\phi_R^2]u := 
(-\Delta)^{\frac{\alpha}{4}}(u\phi_R^2) - \phi_R^2 (-\Delta)^{\frac{\alpha}{4}} u.
\]
Since $\nabla \cdot u = 0$, the pressure can be described as
\[
p = R_i R_j (u_i u_j),
\]
where $R_j$ is the $j$-th Riesz transform. Hence,
\begin{equation} \label{eq:3.2} \tag{3.2}
\|p\|_{L^r(\mathbb{R}^n)} \le C \| u \|_{L^{2r}(\mathbb{R}^n)}^2,
\quad 1<r<+\infty.    
\end{equation}
By H\"older's inequality and \eqref{eq:3.2}, we can estimate
\begin{equation} \label{eq:3.3} \tag{3.3}
\begin{aligned}
I_1+I_2 &\le 
C R^{n-1-\frac{3n}{r}}\|u\|_{L^{r}(A(R))}^3 
+ C R^{n-1-\frac{3n}{r}}\|u\|_{L^{r}(A(R))}\|u\|_{L^{r}(\mathbb{R}^n)}^2 \\
&\le \,C R^{n-1-\frac{3n}{r}}\|u\|_{L^{r}(A(R))}\|u\|_{L^{r}(\mathbb{R}^n)}^2,
\end{aligned}
\end{equation}
where $3 \le r \le \frac{3n}{n-1}$.

By a direct computation
\[
[(-\Delta)^{\frac{\alpha}{4}},\phi_R^2]u
=
\phi_R[(-\Delta)^{\frac{\alpha}{4}},\phi_R]u +
[(-\Delta)^{\frac{\alpha}{4}},\phi_R](\phi_R u),
\]
we may rewrite $I_3$ as
\begin{equation*}
\begin{aligned}
I_3 &= -\int_{B_{2R}} \phi_R(-\Delta)^{\frac{\alpha}{4}} u 
\cdot [(-\Delta)^{\frac{\alpha}{4}},\phi_R]u\, dx 
-\int_{\mathbb R^n} (-\Delta)^{\frac{\alpha}{4}} u 
\cdot [(-\Delta)^{\frac{\alpha}{4}},\phi_R](\phi_R u)\, dx \\
&=:I_4+I_5.
\end{aligned}
\end{equation*}

We begin by establishing fractional commutator estimates for $I_4$, dividing the cases $0<\alpha<2$ and $2<\alpha<4$. This is because the near-field kernel is directly integrable in the former case, whereas in the latter case a first-order cancellation is needed to reduce the singularity.
\begin{proposition} \label{prop:1}
Assume that $u$ is a smooth solution of \eqref{eq:1.1} satisfying the assumptions of Theorem~\ref{thm:1}. For $3\le n \le 5$ and $0<\alpha<2$, we have
\[
I_4 \le
\frac{1}{4} \|(-\Delta)^{\frac{\alpha}{4}}u\|_{L^{2}(B_{2R})}^2
+ C R^{n-\alpha-\frac{2n}{p}} \|u\|_{L^{p}(\mathbb{R}^n)}^2
\]
where $2\le p< \frac{2n}{n-\alpha}$.
\end{proposition}

\begin{proof} [Proof of Proposition~\ref{prop:1}]
For convenience, we set
\[
\gamma := \frac{\alpha}{2}, \quad \Lambda^\gamma := (-\Delta)^{\frac{\gamma}{2}}, 
\quad [\Lambda^\gamma, \phi_R]\,u
=
\Lambda^\gamma(u\phi_R) - \phi_R \Lambda^\gamma u.
\]
Using Lemma~\ref{lem:1}, we may write \([\Lambda^\gamma,\phi_R]u(x)=\mathcal N_R(x)+\mathcal F_R(x),\)
where
\begin{equation*}
\begin{aligned}
\mathcal N_R(x) :=& -C(n,\gamma)\int_{|y|\le 4R}\frac{\big(\phi_R(x+y)-\phi_R(x)\big)u(x+y)}{|y|^{n+\gamma}}\,dy \\
&-C(n,\gamma)\int_{|y|\le 4R}\frac{\big(\phi_R(x-y)-\phi_R(x)\big)u(x-y)}{|y|^{n+\gamma}}\,dy,
\end{aligned}
\end{equation*}
and using $\phi_R(x\pm y)=0$ for $|y|>4R$,
\[
\mathcal F_R(x):=
C(n,\gamma)\,\phi_R(x)\int_{|y|>4R}\frac{u(x+y)+u(x-y)}{|y|^{n+\gamma}}\,dy.
\]
Here we omit the principal value since the integral is integrable for $0<\gamma<1$.

We now estimate the near part $\mathcal N_R$. Using $|\phi_R(x\pm y)-\phi_R(x)|\le C R^{-1}|y|$, we have
\[
|\mathcal N_R(x)|
\le C R^{-1}\int_{|y|\le 4R}
\frac{|u(x+y)|+|u(x-y)|}{|y|^{n-1+\gamma}}\,dy.
\]
Taking the $L^2$-norm and applying Minkowski's inequality, we obtain\
\begin{equation*}
\begin{aligned}
\|\mathcal N_R\|_{L^2(B_{2R})}
&\le C R^{-1}\int_{|y|\le 4R}
\frac{\|u(\,\cdot+y)\|_{L^2(B_{2R})}+\|u(\,\cdot-y)\|_{L^2(B_{2R})}}{|y|^{n-1+\gamma}}\,dy \\
&\le C R^{-1}\int_{|y|\le 4R}
\frac{\|u\|_{L^2(B_{6R})}}{|y|^{n-1+\gamma}}\,dy.
\end{aligned}
\end{equation*}
Using the polar coordinates, we get
\[
\int_{|y|\le 4R}\frac{dy}{|y|^{n-1+\gamma}}
\leq C\int_0^{4R} r^{n-1-(n-1+\gamma)}\,dr
= C\int_0^{4R} r^{-\gamma}\,dr
\leq CR^{1-\gamma}.
\]
Therefore, by Hölder's inequality,
\begin{equation} \label{eq:3.4} \tag{3.4}
\|\mathcal N_R\|_{L^2(B_{2R})}
\le C R^{-\gamma}\|u\|_{L^2(B_{6R})}
\le C R^{\frac{n}{2}-\gamma-\frac{n}{p}}
\|u\|_{L^{p}(\mathbb R^n)},    
\end{equation}
where $2\le p< \frac{2n}{n-2\gamma}$.

Next, we estimate the far part $\mathcal F_R$. Since $|\phi_R(x)|\le 1$ for all $x\in\mathbb{R}^n$, we have
\[
|\mathcal F_R(x)|
\le C\int_{|y|>4R}\frac{|u(x+y)|+|u(x-y)|}{|y|^{n+\gamma}}\,dy.
\]
Applying H\"older's inequality with $\frac{1}{p}+\frac{1}{p'}=1$, we obtain
\begin{equation*}
\begin{aligned}
\int_{|y|>4R}\frac{|u(x\pm y)|}{|y|^{n+\gamma}}\,dy
&\leq
 C\Big(\int_{|y|>4R}|u(x\pm y)|^{p}\,dy\Big)^{\frac{1}{p}}
\Big(\int_{|y|>4R}|y|^{-(n+\gamma)p'}\,dy\Big)^{\frac{1}{p'}} \\
& \le C\|u\|_{L^{p}(\mathbb R^n)}\Big(\int_{|y|>4R}|y|^{-(n+\gamma)p'}\,dy\Big)^{\frac{1}{p'}}.
\end{aligned}
\end{equation*}
Here, we note that for $2\le p < \frac{2n}{n-2\gamma}$,
\[
\Big(\int_{|y|>4R}|y|^{-(n+\gamma)p'}\,dy\Big)^{1/p'}
\leq C \Big(\int_{4R}^\infty r^{n-1-(n+\gamma)p'}\,dr\Big)^{1/p'}
\leq C R^{\frac{n}{p'}-(n+\gamma)}
= R^{-\frac{n}{p}-\gamma}.
\]
Taking the $L^2$-norm, we obtain
\begin{align*} \label{eq:3.5} \tag{3.5}
\|\mathcal F_R\|_{L^2(B_{2R})}
\le
C R^{-\frac{n}{p}-\gamma}\,|B_{2R}|^{\frac{1}{2}}\,
\|u\|_{L^{p}(\mathbb R^n)}
\le C R^{\frac{n}{2}-\gamma-\frac{n}{p}}\|u\|_{L^{p}(\mathbb R^n)}.
\end{align*}
Combining the results \eqref{eq:3.4} and \eqref{eq:3.5}, we have
\begin{equation} \label{eq:3.6} \tag{3.6}
\|[(-\Delta)^{\frac{\alpha}{4}},\phi_R]u\|_{L^2(B_{2R})} 
\le C\|\mathcal N_R + \mathcal F_R\|_{L^2(B_{2R})}
\le CR^{\frac{n}{2}-\frac{\alpha}{2}-\frac{n}{p}} \|u\|_{L^{p}(\mathbb{R}^n)}
\end{equation}
Using \eqref{eq:3.6} and Young's inequality, we obtain
\begin{align*}
\int_{B_{2R}} \phi_R(-\Delta)^{\frac{\alpha}{4}} u 
\cdot [(-\Delta)^{\frac{\alpha}{4}},\phi_R]u\, dx 
& \le C \|(-\Delta)^{\frac{\alpha}{4}}u\|_{L^{2}(B_{2R})}R^{\frac{n}{2}-\frac{\alpha}{2}-\frac{n}{p}} \|u\|_{L^{p}(\mathbb{R}^n)} \\
& \le \frac{1}{4} \|(-\Delta)^{\frac{\alpha}{4}}u\|_{L^{2}(B_{2R})}^2+
CR^{n-\alpha-\frac{2n}{p}} \|u\|_{L^{p}(\mathbb{R}^n)}^2,
\end{align*}
where $2\le p< \frac{2n}{n-\alpha}$.
\end{proof}

\begin{proposition} \label{prop:2}
Assume that $u$ is a smooth solution of \eqref{eq:1.1} satisfying the assumptions of Theorem~\ref{thm:2}. For $5\le n \le 6$ and $2<\alpha<4$, we have
\[
I_4\le \frac{1}{4} \|(-\Delta)^{\frac{\alpha}{4}}u\|_{L^{2}(B_{2R})}^2+
CR^{n-\alpha-\frac{2n}{p}} \|u\|_{L^{p}(\mathbb{R}^n)}^2
+CR^{n+2-\alpha-\frac{2n}{q}} \|\nabla u\|_{L^{q}(\mathbb{R}^n)}^2,
\]
where $2\le p < \frac{2n}{n-\alpha}$ and $2 \le q < \frac{2n}{n+2-\alpha}$.
\end{proposition}

\begin{proof} [Proof of Proposition~\ref{prop:2}]
Similarly, we use the same notation as above. By Lemma~\ref{lem:1}, we get
\begin{equation*}
[\Lambda^\gamma,\phi_R]u(x)
= C(n,\gamma)\,\mathrm{P.V.}\int_{\mathbb{R}^n}
\frac{(\phi_R(x)-\phi_R(y))u(y)}{|x-y|^{n+\gamma}}\,dy.
\end{equation*}
Here, the commutator can be written in the form
$[\Lambda^\gamma, \phi_R]u(x) = T_1(x)+T_2(x),$ where
\[
T_1(x) = C(n,\gamma) \int_{\mathbb{R}^n} \frac{\phi_R(x)-\phi_R(y) 
-\nabla\phi_R(x)\cdot(x-y)\mathbf{1}_{|x-y|\le 2R}} {|x-y|^{n+\gamma}} \,u(y)\,dy,
\]
and
\[
T_2(x) = - C(n,\gamma)\nabla\phi_R(x)\cdot \int_{|x-y|\le 2R} \frac{(u(x)-u(y))(x-y)} {|x-y|^{n+\gamma}} \,dy.
\]
Here we omit the principal value since it is integrable after cancellation for $1<\gamma<2$.

We first estimate $T_1$. Using the first-order Taylor expansion
\[
\left|\phi_R(x)-\phi_R(y) - \nabla \phi_R(x)\cdot (x-y) \right|
\le C R^{-2}|x-y|^2, \quad |x-y|\le 2R,
\]
we have
\begin{align*}
|T_1(x)| \le C R^{-2} \int_{|x-y|\le 2R} \frac{|u(y)|}{|x-y|^{n+\gamma-2}} \,dy
+ C\int_{|x-y|>2R} \frac{|u(y)|}{|x-y|^{n+\gamma}} \,dy  
\end{align*}
Changing variables $z = x-y$, we obtain
\begin{align*}
|T_1(x)| \le CR^{-2} \int_{|z|\le 2R} \frac{|u(x-z)|}{|z|^{n+\gamma-2}}\,dz
+ C\int_{|z|>2R} \frac{|u(x-z)|}{|z|^{n+\gamma}}\,dz.
\end{align*}
Using the same argument as in the estimates of \eqref{eq:3.4} and \eqref{eq:3.5}, we obtain
\begin{equation} \label{eq:3.7} \tag{3.7}
\|T_1\|_{L^2(B_{2R})} \le CR^{\frac{n}{2}-\gamma-\frac{n}{p}} \|u\|_{L^{p}(\mathbb{R}^n)},    
\end{equation}
where $2\le p< \frac{2n}{n-2\gamma}$.

We now estimate $T_2$. Using the mean value theorem 
\[
u(x)-u(y) = \int_0^1 \nabla u\big(x-t(x-y)\big)\cdot(x-y)\,dt,
\]
we get 
\[
|T_2(x)| \le C\,|\nabla\phi_R(x)| \int_0^1 \int_{|x-y|\le 2R} 
\frac{|\nabla u(x-t(x-y))|}{|x-y|^{n+\gamma-2}} \,dy\,dt.
\]
Changing variables $z = x-y$, we obtain
\[
|T_2(x)| \le C\,|\nabla\phi_R(x)| \int_0^1 \int_{|z|\le 2R} 
\frac{|\nabla u(x-tz)|} {|z|^{n+\gamma-2}} \,dz\,dt.
\]
Taking the $L^2$-norm and applying Minkowski's inequality, we obtain
\begin{equation*}
\begin{aligned}
\|T_2\|_{L^2(B_{2R})} &\le
C R^{-1} \int_0^1 \int_{|z|\le 2R} 
\frac{\|\nabla u(\,\cdot - tz)\|_{L^2(B_{2R})}}{|z|^{n+\gamma-2}} \,dz\,dt \\
&\le CR^{-1} \int_{|z|\le 2R} \frac{\|\nabla u\|_{L^2(B_{4R})}}{|z|^{n+\gamma-2}} \,dz.  
\end{aligned}
\end{equation*}
Using the same argument as in the estimates of \eqref{eq:3.4}, we obtain
\begin{equation} \label{eq:3.8} \tag{3.8}
\|T_2\|_{L^2(B_{2R})} \le
CR^{\frac{n}{2}+1-\gamma-\frac{n}{q}} \|\nabla u\|_{L^{q}(\mathbb{R}^n)},    
\end{equation}
where $2\le q < \frac{2n}{n+2-2\gamma}$.

Combining the results \eqref{eq:3.7} and \eqref{eq:3.8}, and using Young's inequality, we obtain
\begin{equation*}
\begin{aligned}
\int_{B_{2R}} &\phi_R(-\Delta)^{\frac{\alpha}{4}} u 
\cdot [(-\Delta)^{\frac{\alpha}{4}},\phi_R]u\, dx \\
& \le C \|(-\Delta)^{\frac{\alpha}{4}}u\|_{L^{2}(B_{2R})}
\left(R^{\frac{n}{2}-\frac{\alpha}{2}-\frac{n}{p}} \|u\|_{L^{p}(\mathbb{R}^n)}
+ CR^{\frac{n}{2}+1-\frac{\alpha}{2}-\frac{n}{q}} \|\nabla u\|_{L^{q}(\mathbb{R}^n)}
\right)\\
& \le \frac{1}{4} \|(-\Delta)^{\frac{\alpha}{4}}u\|_{L^{2}(B_{2R})}^2+
CR^{n-\alpha-\frac{2n}{p}} \|u\|_{L^{p}(\mathbb{R}^n)}^2
+CR^{n+2-\alpha-\frac{2n}{q}}
\|\nabla u\|_{L^{q}(\mathbb{R}^n)}^2,
\end{aligned}
\end{equation*}
where $2\le p < \frac{2n}{n-\alpha}$ and $2 \le q < \frac{2n}{n+2-\alpha}$.
\end{proof}

We now estimate $I_5$ by considering the same cases. Since the support is not localized, we need to control in a different way. Following \cite{26}, we decompose the tail region $\mathbb{R}^n \setminus B_{4R}$ into dyadic annuli and estimate the corresponding contributions separately.

\begin{proposition} \label{prop:3}
Assume that $u$ is a smooth solution of \eqref{eq:1.1} satisfying the assumptions of Theorem~\ref{thm:1}. For $3\le n \le 5$ and $0<\alpha<2$, there exists $R_0>0$ such that for all $R \ge \frac{R_0}{4}$,
\[
I_5 \le \frac{1}{4} \|(-\Delta)^{\frac{\alpha}{4}}u\|_{L^{2}(B_{4R})}^2
+ C R^{n-\alpha-\frac{2n}{p}} \|u\|_{L^{p}(\mathbb{R}^n)}^2
+ CR^{-\frac{\alpha}{2}+\lambda} \|u\|_{L^{2}(B_{2R})},
\]
where $2\le p< \frac{2n}{n-\alpha}$ and $0 \le \lambda < \frac{\alpha}{2}$.
\end{proposition}

\begin{proof} [Proof of Proposition~\ref{prop:3}]
Dividing the domain into $B_{4R}$ and $\mathbb{R}^n\setminus B_{4R}$, we consider
\begin{equation*}
\begin{aligned}
I_5 &= -\int_{B_{4R}} (-\Delta)^{\frac{\alpha}{4}} u 
\cdot [(-\Delta)^{\frac{\alpha}{4}},\phi_R](\phi_R u)\, dx
- \int_{\mathbb{R}^n\setminus B_{4R}} (-\Delta)^{\frac{\alpha}{4}} u 
\cdot [(-\Delta)^{\frac{\alpha}{4}},\phi_R](\phi_R u)\, dx \\
&=: I_6+I_7.
\end{aligned}
\end{equation*}
Setting $v:=\phi_R u$ and applying the same argument as in Proposition~\ref{prop:1} to $v$, we obtain
\begin{equation} \label{eq:3.9} \tag{3.9}
\begin{aligned}
I_6 &\le \frac{1}{4} \|(-\Delta)^{\frac{\alpha}{4}}u\|_{L^{2}(B_{4R})}^2
+ C R^{n-\alpha-\frac{2n}{p}} \|v\|_{L^{p}(\mathbb{R}^n)}^2 \\
&\le \frac{1}{4} \|(-\Delta)^{\frac{\alpha}{4}}u\|_{L^{2}(B_{4R})}^2
+ C R^{n-\alpha-\frac{2n}{p}} \|u\|_{L^{p}(\mathbb{R}^n)}^2,
\end{aligned}
\end{equation}
where $2\le p< \frac{2n}{n-\alpha}$. Here the integration domain is $B_{4R}$ instead of $B_{2R}$, which does not affect the argument.

We now only need to estimate $I_7$. Let $R_k=2^kR$ for $k \in \mathbb{N}$. Then
\[
\mathbb{R}^n\setminus B_{4R}
=
\bigcup_{k\ge2} A(R_k),
\quad
A(R_k)=\{x\in\mathbb{R}^n:R_k\le |x|<2R_k\}.
\]
Since $\phi_R(x)=0$ for $|x|\ge 2R$, for $x \in A(R_k)$ with $k\ge2$, we have
\[
[(-\Delta)^{\frac{\alpha}{4}},\phi_R](\phi_R u)(x)
= C\int_{\mathbb{R}^n}
\frac{\phi_R(x)-\phi_R(y)}{|x-y|^{n+\alpha/2}}\,\phi_R(y)u(y)\,dy
=-C\int_{B_{2R}} \frac{\phi_R(y)^2u(y)}{|x-y|^{n+\alpha/2}}\,dy.
\]
If $x\in A(R_k)$ and $y\in \operatorname{supp}\phi_R\subset B_{2R}$, then for any $k \ge 2$,
\[
\frac{1}{2}R_k \le |x-y| \le 3R_k.
\]
Hence, for $x\in A(R_k)$,
\[
\begin{aligned}
\left|[(-\Delta)^{\frac{\alpha}{4}},\phi_R](\phi_R u)(x)\right|
&\le
C\int_{B_{2R}}
\frac{|u(y)|}{|x-y|^{n+\alpha/2}}\,dy \\
&\le
C\left(\int_{B_{2R}} |u(y)|^2\,dy\right)^{\frac{1}{2}}
\left(\int_{B_{2R}} |x-y|^{-2n-\alpha}\,dy\right)^{\frac{1}{2}} \\
&\le
C\|u\|_{L^2(B_{2R})}R_k^{-n-\frac{\alpha}{2}}|B_{2R}|^{\frac{1}{2}} 
\le
CR^{\frac{n}{2}} R_k^{-n-\frac{\alpha}{2}} \|u\|_{L^2(B_{2R})}.
\end{aligned}
\]
Therefore,
\begin{equation*}
\begin{aligned}
I_7 &\le \sum_{k\ge2} \int_{A(R_k)}
\left|(-\Delta)^{\frac{\alpha}{4}}u(x)\right|
\left|[(-\Delta)^{\frac{\alpha}{4}},\phi_R](\phi_R u)(x)\right| dx \\
&\le C R^{\frac {n}{2}}\|u\|_{L^2(B_{2R})}
\sum_{k\ge2} R_k^{-n-\frac{\alpha}{2}}
\int_{A(R_k)} \left|(-\Delta)^{\frac{\alpha}{4}}u(x)\right| dx \\
&\le C R^{\frac n2}\|u\|_{L^2(B_{2R})}
\sum_{k\ge2}
R_k^{-\frac n2-\frac\alpha2}
\|(-\Delta)^{\frac{\alpha}{4}}u\|_{L^2(A(R_k))}.
\end{aligned}
\end{equation*}
By the assumption
\[
\limsup_{r\to\infty} r^{-\lambda}
\left(\int_{A(r)} \left|(-\Delta)^{\frac{\alpha}{4}}u\right|^2\,dx \right)^{\frac{1}{2}}<+\infty,
\quad 0 \le \lambda < \frac{\alpha}{2},
\]
there exist constants $M>0$ and $R_0>0$ such that for all $r\ge R_0$,
\[
\left(\int_{A(r)} \left|(-\Delta)^{\frac{\alpha}{4}}u\right|^2\,dx \right)^{\frac{1}{2}} \le
M r^\lambda.
\]
Choosing $R>0$ so that $4R\ge R_0$, we have for all $R\ge \frac{R_0}{4}$,
\[
I_7
\le C M R^{\frac {n}{2}}
\|u\|_{L^2(B_{2R})}
\sum_{k\ge2}
R_k^{-\frac {n}{2}-\frac{\alpha}{2}+\lambda}
\le C M R^{-\frac{\alpha}{2}+\lambda}
\|u\|_{L^2(B_{2R})}
\sum_{k\ge2}2^{-k\big(\frac {n}{2}+\frac{\alpha}{2}-\lambda\big)}.
\]
Since $0\le \lambda<\frac{\alpha}{2}<\frac{n+\alpha}{2}$, the series converges. Absorbing $M$ and the value of the series into the constant $C$, we obtain
\begin{equation} \label{eq:3.10} \tag{3.10}
I_7 \le CR^{-\frac{\alpha}{2}+\lambda} \|u\|_{L^2(B_{2R})}.    
\end{equation}
Combining the results \eqref{eq:3.9} and \eqref{eq:3.10}, we can complete the proof.
\end{proof}

\begin{proposition} \label{prop:4}
Assume that $u$ is a smooth solution of \eqref{eq:1.1} satisfying the assumptions of Theorem~\ref{thm:2}. For $5\le n \le 6$ and $2<\alpha<4$, there exists $R_0>0$ such that for all $R \ge \frac{R_0}{4}$,
\begin{equation*}
\begin{aligned}
I_5 \le& \, \frac{1}{4} \|(-\Delta)^{\frac{\alpha}{4}}u\|_{L^{2}(B_{4R})}^2
+ CR^{n-\alpha-\frac{2n}{p}} \|u\|_{L^{p}(\mathbb{R}^n)}^2 \\
& + CR^{n+2-\alpha-\frac{2n}{q}} \|\nabla u\|_{L^{q}(\mathbb{R}^n)}^2
+ CR^{-\frac{\alpha}{2}+\lambda} \|u\|_{L^{2}(B_{2R})},
\end{aligned}
\end{equation*}
where $2\le p< \frac{2n}{n-\alpha}$, $2 \le q < \frac{2n}{n+2-\alpha}$ and $0 \le \lambda < \frac{\alpha}{2}$.
\end{proposition}

\begin{proof} [Proof of Proposition~\ref{prop:4}]
Similarly, we decompose 
\begin{equation*}
\begin{aligned}
I_5 &= -\int_{B_{4R}} (-\Delta)^{\frac{\alpha}{4}} u 
\cdot [(-\Delta)^{\frac{\alpha}{4}},\phi_R](\phi_R u)\, dx
- \int_{\mathbb{R}^n\setminus B_{4R}} (-\Delta)^{\frac{\alpha}{4}} u 
\cdot [(-\Delta)^{\frac{\alpha}{4}},\phi_R](\phi_R u)\, dx \\
&=: I_6+I_7.
\end{aligned}
\end{equation*}
Following the same argument as in the proof of Proposition~\ref{prop:3}, we set $v:=\phi_R u$ and apply Proposition~\ref{prop:2} to obtain
\begin{equation} \label{eq:3.11} \tag{3.11}
I_6 \le \frac{1}{4} \|(-\Delta)^{\frac{\alpha}{4}}u\|_{L^{2}(B_{4R})}^2
+ C R^{n-\alpha-\frac{2n}{p}} \|u\|_{L^{p}(\mathbb{R}^n)}^2
+ CR^{n+2-\alpha-\frac{2n}{q}} \|\nabla u\|_{L^{q}(\mathbb{R}^n)}^2,
\end{equation}
where $2\le p< \frac{2n}{n-\alpha}$ and $2 \le q < \frac{2n}{n+2-\alpha}$. 

Since $I_7$ is independent of the range of $\alpha$, the same estimate used in \eqref{eq:3.10} for all $R \ge \frac{R_0}{4}$, together with \eqref{eq:3.11}, yields the desired bound.
\end{proof}

\begin{proof} [Proof of Theorem~\ref{thm:1}]
Using Proposition~\ref{prop:1} and Proposition~\ref{prop:3}, we can estimate $I_3$ as follows: for all $R \ge \frac{R_0}{4}$, 
\begin{equation} \label{eq:3.12} \tag{3.12}
I_3 \le \frac{1}{2} \|(-\Delta)^{\frac{\alpha}{4}}u\|_{L^{2}(B_{4R})}^2
+ C R^{n-\alpha-\frac{2n}{p}} \|u\|_{L^{p}(\mathbb{R}^n)}^2
+ CR^{-\frac{\alpha}{2}+\lambda} \|u\|_{L^{2}(B_{2R})}.
\end{equation} 
Substituting \eqref{eq:3.3} and \eqref{eq:3.12} into \eqref{eq:3.1}, and applying Lemma~\ref{lem:2} to the function
\[
f(r) := \int_{B_r} \bigl|(-\Delta)^{\frac{\alpha}{4}} u\bigr|^2 \, dx,
\]
we obtain the following Caccioppoli-type inequality:
\begin{equation} \label{eq:3.13} \tag{3.13}
\begin{aligned}
\int_{B_{R}} \bigl|(-\Delta)^{\frac{\alpha}{4}} u\bigr|^2 
\le& \, C R^{\,n-\alpha-\frac{2n}{p}}\|u\|_{L^{p}(\mathbb{R}^n)}^2
+ CR^{-\frac{\alpha}{2}+\lambda} \|u\|_{L^{2}(B_{2R})} \\
&+ C R^{\,n-1-\frac{3n}{r}} \|u\|_{L^{r}(A(R))} \|u\|_{L^{r}(\mathbb{R}^n)}^2,
\end{aligned}
\end{equation}
where $2\le p< \frac{2n}{n-\alpha}$ and $3\le r \le \frac{3n}{n-1}$. 

We now observe that if $u \in L^2(\mathbb{R}^n)$, $u \in L^p(\mathbb{R}^n)$, and $u \in L^r(\mathbb{R}^n)$ for some $2\le p< \frac{2n}{n-\alpha}$ and $3\le r \le \frac{3n}{n-1}$, then we can conclude that $u\equiv0$ as letting $R \to +\infty$. Once the $L^2$-integrability of $u$ is obtained, interpolation between
$L^2(\mathbb R^n)$ and the assumed integrability conditions $L^q(\mathbb R^n)$ gives $u \in L^{q_*}(\mathbb R^n)$ for every $2\le q_*\le q$. Therefore, it remains only to verify that $u\in L^2(\mathbb R^n)$ under our assumptions $L^q(\mathbb R^n)$. Indeed, since the smoothness and the uniform decay condition of $u$ imply $u\in L^\infty(\mathbb R^n)$, one may interpolate between $L^2(\mathbb R^n)$ and $L^\infty(\mathbb R^n)$ instead of between $L^2(\mathbb R^n)$ and $L^q(\mathbb R^n)$.

Since $u$ is smooth and satisfies the uniform decay condition, we have $u \in L^\infty(\mathbb{R}^n)$. 

\noindent
\textbf{Case $\alpha=1$.} Suppose that $u \in L^q(\mathbb{R}^n)$ for some $3 \le q < +\infty$.
By Lemma~\ref{lem:4}, it follows that 
\[
u \in L^{q \cdot 2^{-m}}(\mathbb{R}^n)
\]
for any $m\in\mathbb{N}$ such that $1<q\cdot 2^{-m} \leq 2$. If we take the smallest $m$, say $m_0$, by interpolation between $L^q(\mathbb{R}^n)$ and $L^{q \cdot 2^{-m_0}}(\mathbb{R}^n)$, we get $u \in L^2(\mathbb{R}^n)$.

\noindent
\textbf{Case $1<\alpha<\min\!\left\{\frac{n+2}{3},\,2\right\}$.} For convenience, we formulate the exponent relation 
\[
\frac{1}{q}=\frac{2}{r}-\frac{\alpha-1}{n}
\]
arising from Lemma~\ref{lem:4} as an iteration of a Möbius transformation. This iteration lowers the integrability exponent, allowing us to pass from a higher integrability assumption to the lower exponent.

Suppose that $u \in L^q(\mathbb{R}^n)$ for some $3 \le q < \frac{n}{\alpha-1}$. Let 
\[
F(x):=\frac{nx}{2n-(\alpha-1)x}, \quad M:=\begin{pmatrix}n & 0 \\1-\alpha & 2n\end{pmatrix}.
\]
It is well known that Möbius transformations are closed under composition and can be represented by $2\times2$ matrices up to scaling; see \cite{1}. In particular, since $F$ is a Möbius transformation, its $m$-th iterate is given by
\[
F^{(m)}(x)=\frac{A_m x}{C_mx+D_m}, \quad M^m=
\begin{pmatrix}
A_m & 0\\
C_m & D_m
\end{pmatrix}.
\]
We can compute $M^m$ explicitly:
\[
M^m=\begin{pmatrix}
n^m & 0 \\
(1-\alpha)\displaystyle\sum_{k=0}^{m-1} n^{\,m-1-k}(2n)^{k} & (2n)^m
\end{pmatrix}=
\begin{pmatrix}
n^m & 0\\
(1-\alpha)n^{m-1}(2^m-1) & (2n)^m
\end{pmatrix}.
\]
Hence,
\[
F^{(m)}(x)
=\frac{nx}{n\cdot2^m+(1-\alpha)(2^m-1)x}.
\]
Note that $F(q_\ast)=q_\ast$ for \(q_\ast:=\frac{n}{\alpha-1}.\) Moreover, a direct computation shows that
\[
F(x)-x=\frac{x\bigl((\alpha-1)x-n\bigr)}{2n+(1-\alpha)x}<0 \quad \text{and} \quad F'(x)=\frac{2n^2}{\bigl(2n+(1-\alpha)x\bigr)^2}>0
\quad \text{for all } 0<x<q_\ast.
\]
Hence, $F^{(m)}(x)\to 0$ as $m\to+\infty$ for any $0<x<q_\ast$.

Let $q_{m+1}=F(q_m)$. Since $q_{m+1}$ and $q_m$ satisfy \(\frac{1}{q_{m+1}}=\frac{2}{q_m}-\frac{\alpha-1}{n}\), by Lemma~\ref{lem:4}, we get
\[
\|u\|_{L^{q_{m+1}}(\mathbb R^n)}\le C\|u\|_{L^{q_m}(\mathbb R^n)}^{2}.
\]
Choose $3\le q_0<q_\ast$ so that $q_m \to 0$ as $m \to +\infty$. Then, there exists the smallest number $m_0\in\mathbb{N}$ such that $1<q_{m_0} \le 2$. By the monotonicity of $F$, we have
\[
1 < \frac{n}{n-(\alpha-1)} = F(2) < F(q_{{m_0}-1}) = q_{m_0} \le 2, \quad 1<\alpha<\min\!\left\{\frac{n+2}{3},\,2\right\}.
\]
Therefore, by iterating finitely many times, we obtain
\[
u\in L^{q_{m_0}}(\mathbb R^n), \quad 1 < q_{m_0} \le 2.
\]
Suppose that $u \in L^q(\mathbb{R}^n)$ for some $3 \le q < \frac{n}{\alpha-1}$. By interpolation between $L^q(\mathbb{R}^n)$ and $L^{q_{m_0}}(\mathbb R^n)$, we obtain $u \in L^2(\mathbb{R}^n)$. Here, we note from the relation $q_{m+1}=\frac{2nq_m}{(\alpha-1)q_m+n}$ that the condition $q_{m+1}<2q_m$ can be ignored for $\alpha>1$, since
\[
\frac{2nq_m}{(\alpha-1)q_m+n}<2q_k
\quad \Longleftrightarrow \quad
0<(\alpha-1)q_m.
\]
Moreover, \(q_\ast=\frac{n}{\alpha-1}<\frac{2n}{\alpha-1},\) so the M\"obius mapping $F$ can be iterated finitely many times without any problem.
\end{proof}

\begin{proof} [Proof of Theorem~\ref{thm:2}]
As in the proof of Theorem~\ref{thm:1}, we derive a Caccioppoli-type inequality and then use a bootstrap argument to obtain the required integrability condition. 

Using Proposition~\ref{prop:2} and Proposition~\ref{prop:4}, we have for all $R \ge \frac{R_0}{4}$,
\begin{equation} \label{eq:3.14} \tag{3.14}
\begin{aligned}
I_3 \le& \, \frac{1}{4} \|(-\Delta)^{\frac{\alpha}{4}}u\|_{L^{2}(B_{4R})}^2
+ C R^{n-\alpha-\frac{2n}{p}} \|u\|_{L^{p}(\mathbb{R}^n)}^2 \\
&+ CR^{n+2-\alpha-\frac{2n}{q}} \|\nabla u\|_{L^{q}(\mathbb{R}^n)}^2
+ CR^{-\frac{\alpha}{2}+\lambda} \|u\|_{L^{2}(B_{2R})}.
\end{aligned}
\end{equation} 
Substituting \eqref{eq:3.3} and \eqref{eq:3.14} into \eqref{eq:3.1}, and applying Lemma~\ref{lem:2} to the function $f(r)$ defined as before, we obtain
\begin{equation} \label{eq:3.15} \tag{3.15}
\begin{aligned}
\int_{B_{R}} \bigl|(-\Delta)^{\frac{\alpha}{4}} u\bigr|^2 
\le& \, C R^{\,n-\alpha-\frac{2n}{p}}\|u\|_{L^{p}(\mathbb{R}^n)}^2
+ CR^{n+2-\alpha-\frac{2n}{q}} \|\nabla u\|_{L^{q}(\mathbb{R}^n)}^2 \\
&+ C R^{\,n-1-\frac{3n}{r}} \|u\|_{L^{r}(A(R))} \|u\|_{L^{r}(\mathbb{R}^n)}^2
+ CR^{-\frac{\alpha}{2}+\frac{\lambda}{2}} \|u\|_{L^{2}(B_{2R})} ,
\end{aligned}
\end{equation}
where $2\le p< \frac{2n}{n-\alpha}$, $2 \le q < \frac{2n}{n+2-\alpha}$ and $3 \le r \le \frac{3n}{n-1}$. 

Here, we need to verify that $u \in L^2(\mathbb{R}^n)$ and $\nabla u \in L^q(\mathbb{R}^n)$ for some $2 \le q < \frac{2n}{n+2-\alpha}$ under our assumptions in order to conclude that $u\equiv 0$ by letting $R \to +\infty$. To control the gradient term, we may write
\[
\nabla u = \nabla(-\Delta)^{-\frac{\alpha}{2}}
\,\mathbb P \nabla\cdot\big (u\otimes u).
\]
By the same argument as in Lemma~\ref{lem:4}, we obtain
\begin{equation} \label{eq:3.16} \tag{3.16}
\|u\|_{L^{\frac{ns}{n-s}}(\mathbb{R}^n)}
\le C\|\nabla u\|_{L^s(\mathbb{R}^n)}
\le C\|(-\Delta)^{\frac{2-\alpha}{2}} (u\otimes u)\|_{L^s(\mathbb{R}^n)}
\le C\|u\|_{L^r(\mathbb{R}^n)}^2,    
\end{equation}
where $\frac{1}{s}=\frac{2}{r}-\frac{\alpha-2}{n}$ with $2<r<\frac{2n}{\alpha-1}$. Writing
$s=\frac{nr}{2n+(2-\alpha)r}$, we compute
\[
\frac{ns}{n-s}
= \frac{nr}{2n+(1-\alpha)r}.
\]
Since $\frac{ns}{n-s}<r$ whenever $2<r<\frac{n}{\alpha-1}$, we may apply the bootstrap argument as before. 

Setting $r_{m+1}=\frac{ns}{n-s}$ with $\frac{1}{s}=\frac{2}{r_m}-\frac{\alpha-2}{n}$, we consider the relation $r_m=\frac{nr_{m-1}}{2n+(1-\alpha)r_{m-1}}$ with $\frac{2n}{n-2} \le r_0 < \frac{n}{\alpha-1}$. Explicitly, we can write 
\[
r_m = \frac{n r_0}{n \cdot 2^m + (1-\alpha) r_0 (2^m - 1)}.
\]
Since $r_m$ satisfies the same recursive relation as $q_m$ defined above, using the same mapping $F$, we can find the smallest number $m_0\in\mathbb{N}$ such that $q_{m_0} \le 3$. By the monotonicity of $F$, we have
\[
2 < \frac{3n}{2n-3(\alpha-1)} = F(3) < F(r_{{m_0}-1}) = r_{m_0} \le 3
\quad\text{for } 2<\alpha<\frac{n+2}{3}.
\]
Therefore, if necessary, by iterating finitely many times, we obtain
\[
u\in L^{r_{m_0}}(\mathbb R^n), \quad 2 < r_{m_0} \le 3.
\]
Suppose that $\nabla u \in L^s(\mathbb{R}^n)$ where $2 \le s < \frac{n}{\alpha}$. By the Sobolev embedding, we get $u \in L^{r(s)}(\mathbb{R}^n)$ where $\frac{2n}{n-2} \le r(s) < \frac{n}{\alpha-1}$, which coincides with the admissible range of $r_0$. Recall that we obtain $r_{m_0}$ from the Sobolev embedding for $\nabla u \in L^{n r_{m_0}/(n + r_{m_0})}$. Since $2 < r_{m_0} \le 3$, we note that \(1 < {n r_{m_0}/(n + r_{m_0})} \le 2\). By interpolation between $L^{{n r_{m_0}/(n + r_{m_0})}}(\mathbb{R}^n)$ and $L^s(\mathbb{R}^n)$, we obtain $\nabla u \in L^q(\mathbb{R}^n)$ for some $2 \le q < \frac{2n}{n+2-\alpha}$. Using the same mapping $F$ as in the proof of Theorem~\ref{thm:1} and following the same argument, we also have $u \in L^2(\mathbb{R}^n)$.
\end{proof}

\section*{Acknowledgements}
Jihoon Lee's work was supported by the National Research Foundation of Korea (NRF) grant funded by the Korea government (MSIT) (No.\ NRF-RS-2026-25472425).

\begin{appendices}
\section{Stationary fractional MHD in $\mathbb{R}^3$}
Extending our idea to a coupled model, we consider the stationary fractional magnetohydrodynamic system in $\mathbb{R}^3$:
\begin{equation} \label{eq:A.1} \tag{A.1}
\begin{cases}
(-\Delta)^{\frac{\alpha}{2}} u + (u \cdot \nabla) u - (B \cdot \nabla) B + \nabla p = 0, \\
(-\Delta)^{\frac{\beta}{2}} B + (u \cdot \nabla) B - (B \cdot \nabla) u = 0, \\
\nabla \cdot u =\nabla \cdot B = 0,
\end{cases}
\end{equation}
with the uniform decay conditions
\[
u(x)\to 0,\quad B(x)\to 0\quad \text{as} \quad |x|\to+\infty\ .
\]
Here, $u = (u_1(x), u_2(x), u_3(x))$ denotes the velocity field of the fluid, $B = (B_1(x), B_2(x), B_3(x))$ denotes the magnetic field, and $p = p(x)$ denotes the pressure. For convenience, we divide the domain $(\alpha,\beta)$ into two regions, denoted by $A_1$ and $A_2$, where
\[
A_1 := \bigl\{1 \le \alpha < 5 - 2\beta,\; \tfrac{5}{3} < \beta < 2 \bigr\}, \quad
A_2 := \bigl\{1 \le \alpha \le \tfrac{5}{3},\; 1 \le \beta \le \tfrac{5}{3} \bigr\}.
\]

We prove the Liouville-type theorems in the region $A_1$, where $\overline{A_1}\setminus \left\{\left(\frac{5}{3},\frac{5}{3}\right)\right\}$ is the remaining admissible region for $(\alpha,\beta)$ associated with the result of Chae and Lee \cite{11}. In particular, they proved triviality under interpolation-type integrability conditions: 
\begin{equation} \label{eq:A.2} \tag{A.2}
(u,B) \in L^{p_1}(\mathbb R^3) \times L^{p_2}(\mathbb R^3),\quad3 \le p_1 \le \frac{9}{2},\quad 3 \le p_2 \le \frac{6p_1}{2p_1-3},
\end{equation}
which satisfies \(2 - \frac{3}{p_1} - \frac{6}{p_2} \le 0.\) By the Sobolev embeddings
\[
\dot H^{\frac{\alpha}{2}}(\mathbb R^3)\hookrightarrow 
L^{\frac{6}{3-\alpha}}(\mathbb R^3),
\quad
\dot H^{\frac{\beta}{2}}(\mathbb R^3)\hookrightarrow 
L^{\frac{6}{3-\beta}}(\mathbb R^3),
\]
the admissible $(\alpha,\beta)$ associated with \eqref{eq:A.2} are given by $\overline{A_1}
\cup A_2$.

For the region $A_2$, Zeng \cite{42} proved the following: let $0<\alpha,\beta\le 2$ and $\vec p_j=(p_{j,1},p_{j,2},p_{j,3})$, $\vec q_j=(q_{j,1},q_{j,2},q_{j,3})$ with $p_{j,l},q_{j,l}\in[3,\infty)$ for $j,l=1,2,3$. Assume that
\[
\sum_{l=1}^3 \frac{1}{p_{j,l}} \ge \frac{2}{3}, \quad
\sum_{l=1}^3 \frac{1}{q_{j,l}} \ge \frac{2}{3},
\quad j=1,2,3.
\]
If $(u,B)\in \dot H^\frac{\alpha}{2}(\mathbb R^3)\times \dot H^\frac{\beta}{2}(\mathbb R^3)$ is a smooth solution such that $(u_j,B_j)\in L^{\vec p_j}(\mathbb R^3)\times L^{\vec q_j}(\mathbb R^3)$ for $j=1,2,3$, then $u\equiv B \equiv0$. As a corollary, any smooth solution $(u,B)\in \dot H^{\frac{\alpha}{2}}(\mathbb{R}^3)\times \dot H^\frac{\beta}{2}(\mathbb{R}^3)$ is trivial for $1\le \alpha,\beta \le \frac{5}{3}$. For the stationary MHD system in $\mathbb{R}^3$, several Liouville-type results have been established under various conditions, and we refer to \cite{7,8,9,17,18,37}.

\begin{theorem} \label{thm:A}
Let $(u,B)$ be a smooth solution of \eqref{eq:A.1}. Assume that for some $0 \le \lambda_1 < \frac{\alpha}{2}$ and $0 \le \lambda_2 < \frac{\beta}{2}$,
\[
\limsup_{r\to\infty} \left(r^{-\lambda_1}\left(\int_{A(r)}\left|(-\Delta)^{\frac{\alpha}{4}} u\right|^2 \, dx\right)^{\frac{1}{2}}
+
r^{-\lambda_2}\left(\int_{A(r)}\left|(-\Delta)^{\frac{\beta}{4}} B\right|^2 \, dx\right)^{\frac{1}{2}}\right)
< +\infty.
\]
Suppose further that one of the following conditions holds:
\[
\begin{aligned}
\text{(i)}\quad & 1\le \alpha< 5-2\beta, \ \frac{5}{3} < \beta < 2 : 
&& (u,B) \in L^{\frac{6}{3-\alpha}}(\mathbb{R}^3)\times L^{\frac{6}{3-\beta}}(\mathbb{R}^3), \\
\text{(ii)}\quad &  \alpha=5-2\beta, \ \frac{5}{3} < \beta \le 2: && (u,B) \in L^{\frac{6}{3-\alpha}-\varepsilon}(\mathbb{R}^3) \times L^{\frac{6}{3-\beta}}(\mathbb{R}^3), \quad 0<\varepsilon \ll 1. \\
\end{aligned}
\]
Then $u\equiv B \equiv0$.

In particular, any smooth solution $(u,B) \in \dot H^{\frac{\alpha}{2}}\times \dot H^{\frac{\beta}{2}}$ of \eqref{eq:A.1} is trivial for $1\le \alpha< 5-2\beta$ with $\frac{5}{3} < \beta <2$.
\end{theorem}

\begin{proof} [Proof of Theorem~\ref{thm:A}] The argument follows as in the case of \eqref{eq:1.1}, based on the same kernel estimates for $u$ and $B$ and a bootstrap argument. We first estimate $L^q$ bounds for $u$ and $B$.

\begin{lemma} \label{lem:A}
Let $(u,B)$ be a smooth solution of \eqref{eq:A.1} satisfying $(u,B) \in L^r(\mathbb{R}^3) \times L^r(\mathbb{R}^3)$, where $r$ is specified below. Then the following estimates hold:
\[
\begin{aligned}
\text{(i)} \ &\alpha=1:  
&& \|u \|_{L^q(\mathbb{R}^3)}
\le C\|u\|_{L^{r}(\mathbb{R}^3)}^2 + C\|B\|_{L^{r}(\mathbb{R}^3)}^2,
\quad r=2q,
\quad 1<q<+\infty, \\
\text{(ii)} \ &\beta=1:  
&& \|B \|_{L^q(\mathbb{R}^3)}
\le C\|u\otimes B\|_{L^{r}(\mathbb{R}^3)},
\quad r=q, \quad 1<q<+\infty, \\
\text{(iii)} \ & 1<\alpha<2: 
&& \|u\|_{L^q(\mathbb{R}^3)}
\le C\|u\|_{L^{r}(\mathbb{R}^3)}^2+C\|B\|_{L^{r}(\mathbb{R}^3)}^2,
\quad
\frac{1}{q}=\frac{2}{r}-\frac{\alpha-1}{3}, \\
\text{(iv)} \ & 1<\beta<2: 
&& \|B\|_{L^q(\mathbb{R}^3)}
\le C\|u\otimes B\|_{L^{r}(\mathbb{R}^3)},
\quad
\frac{1}{q}=\frac{1}{r}-\frac{\beta-1}{3}.
\end{aligned}
\]
\end{lemma}

\begin{proof} [Proof of Lemma~\ref{lem:A}]
Let $(u,B)$ be a smooth solution of \eqref{eq:A.1}. Then, we may write
\[
u= (-\Delta)^{-\frac{\alpha}{2}}
\,\mathbb P \nabla\cdot\big(u\otimes u - B\otimes B\big), \quad
B
= (-\Delta)^{-\frac{\beta}{2}}
\,\mathbb  \nabla\cdot\big(B\otimes u - u\otimes B\big).
\]
Using the $L^q$-boundedness of the Leray projection and the Riesz transforms, together with Lemma~\ref{lem:3}, we can complete the proof by the same argument as in Lemma~\ref{lem:4}.
\end{proof}

Multiplying equation~\eqref{eq:A.1} by $\phi_R^2 u$ and $\phi_R^2 B$, respectively and integrating, we get
\begin{equation}\label{eq:A.3} \tag{A.3}
\begin{aligned}
\int_{B_{R}} &|(-\Delta)^{\frac{\alpha}{4}}u|^2
+|(-\Delta)^{\frac{\beta}{4}}B|^2\,dx \\
=& \, 
\frac{1}{2} \int_{A(R)} \left(|u|^2 + |B|^2 \right) u \cdot \nabla (\phi_R^2) \, dx 
- \int_{A(R)} (u \cdot B)\, (B \cdot \nabla (\phi_R^2)) \, dx  \\
&- \int_{A(R)} (pu) \cdot \nabla (\phi_R^2) \, dx
-\int_{B_{R}} (-\Delta)^{\frac{\alpha}{4}} u \cdot
[(-\Delta)^{\frac{\alpha}{4}},\phi_R^2]u\, dx \\
&- \int_{B_{R}} (-\Delta)^{\frac{\beta}{4}} B \cdot
[(-\Delta)^{\frac{\beta}{4}},\phi_R^2]B\, dx =:\sum_{i=1}^{5} K_i
\end{aligned}
\end{equation}
Since $\nabla \cdot u = \nabla \cdot B = 0$, we can estimate the pressure
\begin{equation} \label{eq:A.4} \tag{A.4}
\|p\|_{L^q(\mathbb{R}^n)}
\le  C \| u \|_{L^{2q}(\mathbb{R}^n)}^2+C\| B \|_{L^{2q}(\mathbb{R}^n)}^2 ,
\quad 1<q<+\infty.    
\end{equation}
By H\"older's inequality and \eqref{eq:A.4}, we can estimate
\begin{equation} \label{eq:A.5} \tag{A.5}
\begin{aligned}
&K_1+K_2+K_3 \\
&\le C R^{n-1-\frac{3n}{r}}\|u\|_{L^{r}(A(R))}^3
+ C R^{n-1-\frac{n}{q_1}-\frac{2n}{q_2}} 
\|u\|_{L^{q_1}(A(R))}\|B\|_{L^{q_2}(A(R))}^2 \\
& \quad + C R^{n-1-\frac{3n}{r}}\|u\|_{L^{q_1}(A(R))}\|u\|_{L^{r}(\mathbb{R}^n)}^2
+ C R^{n-1-\frac{n}{q_1}-\frac{2n}{q_2}} 
\|u\|_{L^{q_1}(A(R))}\|B\|_{L^{q_2}(\mathbb{R}^n)}^2 \\
& \le C R^{n-1-\frac{3n}{r}}\|u\|_{L^{q_1}(A(R))}\|u\|_{L^{r}(\mathbb{R}^n)}^2
+ C R^{n-1-\frac{n}{q_1}-\frac{2n}{q_2}} 
\|u\|_{L^{q_1}(A(R))}\|B\|_{L^{q_2}(\mathbb{R}^n)}^2,
\end{aligned}
\end{equation}
where $q_1,q_2 \ge 3$ with $\frac{1}{q_1}+\frac{2}{q_2}\ge \frac{n-1}{n}$ and $3 \le r \le \frac{3n}{n-1}$.

Using the same argument as in the proof of Theorem~\ref{thm:1} for $K_4$ and $K_5$, we obtain the following Caccioppoli-type inequality: for all $R \ge \frac{R_0}{4}$,
\begin{equation} \label{eq:A.6} \tag{A.6}
\begin{aligned}
\int_{B_{R}} &\bigl|(-\Delta)^{\frac{\alpha}{4}} u\bigr|^2 
+ \bigl|(-\Delta)^{\frac{\beta}{4}} B\bigr|^2 dx \\
\le & \,
C R^{\,3-\alpha-\frac{6}{p_1}} \, \|u\|_{L^{p_1}(\mathbb{R}^3)}^2
+ C R^{\,3-\beta-\frac{6}{p_2}} \, \|B\|_{L^{p_2}(\mathbb{R}^3)}^2
+ CR^{-\frac{\alpha}{2}+\lambda_1} \|u\|_{L^{2}(B_{2R})} \\
&+ CR^{-\frac{\alpha}{2}+\lambda_2} \|B\|_{L^{2}(B_{2R})} 
+ C R^{\,2-\frac{3}{q_1}-\frac{6}{q_2}}
\, \|u\|_{L^{q_1}(A(R))}
\, \|B\|_{L^{q_2}(\mathbb{R}^3)}^2 \\
&+ C R^{\,2-\frac{9}{r}}
\, \|u\|_{L^{r}(A(R))}
\, \|u\|_{L^{r}(\mathbb{R}^3)}^2,
\end{aligned}
\end{equation}
where $2 \le p_1 < \frac{6}{3-\alpha}$, $2 \le p_2 < \frac{6}{3-\beta}$, $q_1,q_2 \ge 3$ with 
$\frac{1}{q_1}+\frac{2}{q_2}\ge \frac{2}{3}$ and $3 \le r \le \frac{9}{2}$.

Let $q_m$ be defined by
\[
\frac1{q_m}=\frac1p+\frac1{r_m},
\quad
\frac1{r_{m+1}}=\frac1{q_m}-\frac{\beta-1}{3}, \quad r_0=\frac{6}{3-\beta}.
\]
Then, we obtain the relation
\[
\frac1{r_{m+1}}=\frac1{r_m}+c,
\quad
c:=\frac1p-\frac{\beta-1}{3},
\]
so that
\[
\frac1{r_m}=\frac1{r_0}+cm.
\]
\noindent
\textbf{Case $\alpha=1$, $\frac{5}{3}<\beta<2$.} Take $p=3$. By Lemma~\ref{lem:A}, we have
\[
\|B\|_{L^{\frac{6}{7-3\beta}}(\mathbb R^3)}
\le C\|u\otimes B\|_{L^{\frac{6}{5-\beta}}(\mathbb R^3)}
\le C\|u\|_{L^3(\mathbb R^3)}\|B\|_{L^{\frac{6}{3-\beta}}(\mathbb R^3)}.
\]
Repeating the same argument yields
\[
\|B\|_{L^{\frac{6}{11-5\beta}}(\mathbb R^3)}
\le C\|u\otimes B\|_{L^{\frac{2}{3-\beta}}(\mathbb R^3)}
\le C\|u\|_{L^3(\mathbb R^3)}\|B\|_{L^{\frac{6}{7-3\beta}}(\mathbb R^3)}.
\]
Iterating this procedure, we obtain
\[
\|B\|_{L^{r_m}(\mathbb R^3)}
\le C\|u\|_{L^3(\mathbb R^3)}\|B\|_{L^{r_{m-1}}(\mathbb R^3)},
\quad
r_m=\frac{6}{3-\beta+2m(2-\beta)}.
\]
Note that $r_m\to 0$ as $m\to+\infty$. Moreover, the condition \(1< r_m\le2\) implies that \(\frac{\beta}{2(2-\beta)}< m\le\frac{3+\beta}{2(2-\beta)}\). Since $\frac{3+\beta}{2(2-\beta)}-\frac{\beta}{2(2-\beta)}>1$ and $\frac{\beta-1}{2(2-\beta)}>0$ for $\frac{5}{3}<\beta<2$, there exists the smallest number $m_0\in\mathbb N$ such that \(1< r_{m_0}\le2.\) Therefore, by iterating finitely many times, we obtain
\[
B\in L^{r_{m_0}}(\mathbb R^n), \quad 1 < r_{m_0} \le 2.
\]
By interpolation between $L^{r_{m_0}}(\mathbb R^3)$ and $L^{\frac{6}{3-\beta}}(\mathbb R^3)$ for $B$, we conclude that $B\in L^2(\mathbb R^3)$. Applying Lemma~\ref{lem:A} to $u$ with $L^3(\mathbb R^3)$, we obtain
\[ \|u\|_{L^{\frac{3}{2}}(\mathbb R^3)}\le C\|u\|_{L^3(\mathbb R^3)}^2 + C\|B\|_{L^3(\mathbb R^3)}^2 <+\infty.
\]
Interpolating once more between $L^{\frac{3}{2}}(\mathbb R^3)$ and $L^3(\mathbb R^3)$ for $u$ yields
$u\in L^2(\mathbb R^3) \cap L^{p_1}(\mathbb R^3)$. Furthermore, since $u\in L^3(\mathbb R^3)$ and $B\in L^{\frac{6}{3-\beta}}(\mathbb R^3)$ satisfy \(2-\frac{3}{q_1}-\frac{6}{q_2}\le0 \) and $q_1,q_2\ge3$, letting $R\to\infty$ yields \(u\equiv B\equiv0.\)

\noindent
\textbf{Case $1\le \alpha < 5-2\beta$, $\frac{5}{3}<\beta<2$.} Taking $p=\frac{6}{3-\alpha}$, we get \[r_m=\frac{6}{3-\beta+m(5-\alpha-2\beta)},\] where $r_m\to0$ as $m\to+\infty$. Here, the condition \(1< r_m\le2\) implies that \(\frac{\beta}{5-\alpha-2\beta}< m\le\frac{3+\beta}{5-\alpha-2\beta}\). Since $\frac{3+\beta}{5-\alpha-2\beta}-\frac{\beta}{5-\alpha-2\beta}>1$ and $\frac{\beta}{5-\alpha-2\beta}>0$ for $1\le \alpha < 5-2\beta$ with $\frac{5}{3}<\beta<2$, we can conclude that \(u\equiv B\equiv0\) using the same argument.
\\
\noindent
\textbf{Case $\alpha=5-2\beta$, $\frac{5}{3}<\beta\le2$.} Here, $c=\frac{3-\alpha}{6}-\frac{\beta-1}{3}=0$ and therefore the bootstrap argument cannot be applied. If $u\in L^{\frac{6}{3-\alpha}-\varepsilon}$ for sufficiently small $0<\varepsilon \ll 1$, then $c>0$ and hence the bootstrap argument completes the proof.
\end{proof}
\end{appendices}

% ============================================================
% Author information
% ============================================================

\bigskip

\noindent\textbf{Jihoon Lee.}
Department of Mathematics, Chung-Ang University, Seoul 06974, Republic of Korea;
email: \texttt{jhleepde@cau.ac.kr}

\medskip

\noindent\textbf{Juhyeong Lee.}
Department of Mathematics, University of British Columbia, Vancouver, BC V6T 1Z2, Canada;
email: \texttt{juhyeong.lee.math@gmail.com}

\end{document}